\documentclass[a4paper,11pt]{amsart}
\usepackage{amsmath}
\usepackage{amssymb}
\usepackage{amstext}
\usepackage{amsbsy}
\usepackage{amsopn}
\usepackage{amsthm}
\usepackage{amscd}
\usepackage{amsxtra}
\usepackage{amsfonts}
\usepackage{ifthen}
\usepackage{xspace}
\usepackage{fancyheadings}
\usepackage{layout}
\usepackage{hyperref}

\input xy
\xyoption{matrix}
\xyoption{arrow}
\xyoption{curve}
\newdir{ (}{{}*!/-5pt/@^{(}}
\newcommand{\xycenter}[1]{\begin{center}
                          \mbox{\xymatrix{#1}}
                          \end{center}
                         }

\newboolean{xlabels}
\newcommand{\xlabel}[1]{
                        \label{#1}
                        \ifthenelse{\boolean{xlabels}}
                                   {\marginpar[\hfill{\tiny #1}]{{\tiny #1}}}
                                   {}
                       }

\newcommand{\ZZ}{\mathbb{Z}}

\newcommand{\AZ}{\mathbb{A}}

\newcommand{\CC}{\mathbb{C}}
\newcommand{\RR}{\mathbb{R}}
\newcommand{\QQ}{\mathbb{Q}}

\newcommand{\FF}{\mathbb{F}}

\newcommand{\suchthat}{\, | \,}

\newboolean{probleme}
\newcommand{\problem}[1]
           {\ifthenelse{\boolean{probleme}}
                       {{\bf(PROBLEM: #1)\bf}}
                       {}
           }
\newboolean{zukuenftiges}
\newcommand{\zukunft}[1]
           {\ifthenelse{\boolean{zukuenftiges}}
                       {{\bf(AUSBAUM\"OGLICHKEIT: #1)\bf}}
                       {}
           }
\newboolean{extras}
\newcommand{\extra}[1]
           {\ifthenelse{\boolean{extras}}
                       {{\bf EXTRA #1 EXTRA\bf}}
                       {}
           }

\newboolean{ignore}
\newcommand{\ignore}[1]
           {\ifthenelse{\boolean{ignore}}
                       {{\bf IGNORE #1 IGNORE\bf}}
                       {}
           }

\DeclareMathOperator{\codim}{codim}

\DeclareMathOperator{\Gal}{Gal}

\DeclareMathOperator{\Img}{Im}

\DeclareMathOperator{\rank}{rank}

\DeclareMathOperator{\id}{id}

\DeclareMathOperator{\Aff}{Aff}
\DeclareMathOperator{\sing}{sing}

\theoremstyle{plain}
\begingroup
\newtheorem{thm}{Theorem}%[subsection]
\newtheorem{cor}[thm]{Corollary}
\newtheorem{lem}[thm]{Lemma}

\newtheorem{prop}[thm]{Proposition}
\newtheorem{conj}[thm]{Conjecture}

\numberwithin{thm}{subsection} 
\endgroup

\newtheorem*{thm*}{Theorem}
\newtheorem*{conj*}{Conjecture}
\newtheorem*{verm*}{Vermutung}

\theoremstyle{definition}
\newtheorem{defn}[thm]{Definition}
\newtheorem{rem}[thm]{Remark}

\newtheorem{notation}[thm]{Notation}

\newtheorem{experiment}[thm]{Experiment}

\newtheorem{heu}[thm]{Heuristic}
\newtheorem{cau}[thm]{Caution}

\numberwithin{equation}{section}

\newcommand{\nosubsections}{\renewcommand{\thethm}{\thesection.\arabic{thm}}
                            \setcounter{thm}{0}
                           }

\newcommand{\cref}[3]{(\ref{#1}, #2 \ref{#3})}

\date{\today}
\usepackage{amssymb,amsmath}
\setboolean{probleme}{true}
\setboolean{xlabels}{false}

\usepackage{graphicx}

\newcommand{\secemail}{
\setlength{\unitlength}{1pt}
bothmer
\begin{picture}(0,1)
\put(0,0){m}
\put(-5,0){@}
\end{picture}
ath.uni-hannover.de}

\theoremstyle{definition}
\newtheorem{calc}[thm]{Calculation}

\newcommand{\Fp}{\FF_p}
\newcommand{\Anp}{\AZ^n(\Fp)}
\newcommand{\gammatilde}{\tilde{\gamma}}

\begin{document}

\title[A survey of the Poincar{\'e} Center Problem in degree 3 ]
{A survey of the Poincar\'e Center Problem in degree 3 
using finite field heuristics}

\address{Courant Research Centre ''Higher Order Structures''\\
	Mathematisches Institiut\\ 
	University of G\"ottingen\\
         Bunsenstrasse 3-5\\ 
          D-37073 G\"ottingen
         }

\email{\secemail}

\urladdr{http://www.uni-math.gwdg.de/bothmer}

\thanks{Supported by the German Research Foundation 
(Deutsche Forschungsgemeinschaft (DFG)) through 
the Institutional Strategy of the University of G\"ottingen}

\author{Hans-Christian Graf v. Bothmer}
\author{Jakob Kr\"oker}

\begin{abstract}
We compare a heuristic count of components of the center variety in degree
$3$ with the equivalent count obtained from known families. From this comparison we
conjecture that more than 100 unknown components exist. 
\end{abstract} 

\maketitle

\newcommand{\dual}{^*}
\newcommand{\barf}{\bar{f}}
\newcommand{\barg}{\bar{g}}
\newcommand{\barh}{\bar{h}}
\newcommand{\bara}{\bar{a}}
\newcommand{\barJ}{\bar{J}}
\newcommand{\barN}{\bar{N}}
\newcommand{\barH}{\bar{H}}

\newcommand{\impsi}{\overline{\Img \psi}}
\newcommand{\imphi}{\overline{\Img \phi}}
\newcommand{\semi}{Theorem \ref{tSemi}}

\newcommand{\ZZx}{\ZZ[x_1,\dots,x_n]}
\newcommand{\QQx}{\QQ[x_1,\dots,x_n]}
\newcommand{\CCx}{\CC[x_1,\dots,x_n]}
\newcommand{\FFp}{\FF_p}
\newcommand{\FFx}{\FFp[x_1,\dots,x_n]}
\renewcommand{\AA}{\mathbb{A}}

\newcommand{\zoladek}{\.Zo\l\c adek\,}
\newcommand{\zoladeks}{\.Zo\l\c adek's\,}

%%%%%%%%%%%%%%%%%%%%%%%%%%%%%%%%%%%%%%%%%%%%
\section{Introduction}
%%%%%%%%%%%%%%%%%%%%%%%%%%%%%%%%%%%%%%%%%%%%
\nosubsections
In 1885 Poincar\'e asked when the differential equation
\[
y' = - \frac{x + p(x,y)}{y+q(x,y)} =: - \frac{P(x,y)}{Q(x,y)}
\]
with convergent power series $p(x,y)$ and $q(x,y)$ starting with quadratic terms, 
has stable solutions in the neighborhood of the equilibrium solution
$(x,y)=(0,0)$. This means that in such a neighborhood the solutions of the
equivalent plane autonomous system
\begin{align*}
	\dot{x} &= y + q(x,y) = Q(x,y)\\
	\dot{y} &= -x - p(x,y) = -P(x,y)
\end{align*}
are closed curves around $(0,0)$.

Poincar\'e showed that one can iteratively find a formal power series
$F = x^2+y^2+f_3(x,y)+f_4(x,y)+\dots$ such that
\[
	\det \begin{pmatrix} F_x & F_y \\ P & Q \end{pmatrix} = \sum_{j=1}^\infty s_j(x^{2j+2}+y^{2j+2})
\]
with $s_j$ rational polynomials in the coefficients of $P$ and $Q$.
If all $s_j$ vanish, and $F$ is convergent then $F$ is a constant of motion, i.e. its gradient field
satisfies $Pdx+Qdy=0$. Since $F$ starts with $x^2+y^2$ this shows that close to the origin all integral curves are closed and the system is stable. Therefore the $s_j$'s are
called the {\sl focal values} of $Pdx+Qdy$. Often also the notation $\eta_{2j} := s_j$ is used, and the $\eta_i$ are called {\sl Lyapunov quantities}.

Poincar\'e also showed, that if an analytic constant of motion exists, the focal values must vanish.
Later Frommer \cite{Frommer} proved that the systems above are stable if and only if all focal values vanish even without the assumption of convergence of $F$. (Frommer's proof contains a gap which can ben be closed \cite{vWahlGap})

Unfortunately it is in general impossible to check this condition for a given differential equation because
there are infinitely many focal values. In the case where $P$ and $Q$ are polynomials of degree
at most $d$, the $s_j$ are polynomials in finitely many unknowns. Hilbert's Basis Theorem then implies
that the ideal $I_\infty = (s_1,s_2,\dots)$ is finitely generated, i.e there exists an integer $m := m(d)$
such that
\[
		s_1 = s_2 = \dots = s_{m(d)} = 0  \implies s_j = 0 \quad\forall j.
\]
This shows that a finite criterion for stability exists, but due to the indirect proof of
Hilbert's Basis Theorem no value for $m(d)$ is obtained. In fact even today only $m(2)=3$ is known. In \cite{focalValues} we prove $m(3) \ge 13$ for complex centers. 

The proof for $m(2)=3$ is conceptually simple: Compute the first $3$ focal values as polynomials
in the coefficients of $P$ and $Q$ under the assumption $\deg(P)=\deg(Q)=2$. The $3$ polynomials cut out an algebraic variety in the space of all differential equations
of degree $2$. Then decompose, by hand or by computer, this variety into its irreducible components.
For each component prove that all its differential equations have a constant of motion. 

For $d=3$ this approach is not feasible because the polynomials $s_j$
are very large. They involve $14$ variables and are of weighted degree $2j$. 
For example the $s_6$ can be calculated with our script \verb#s6# available at \cite{centerfocusSourceforge} and has already $95760$ terms.  The polynomials $s_j$, $j\ge 7$ are hard to calculate.
Even if we would somehow obtain these polynomials, it is
extremely difficult to decompose the resulting variety into irreducible components. Even $I_5 = (s_1,\dots,s_5)$ can not be decomposed by current systems. So for $d=3$ only partial results are known, for example  \cite{ZoladekRomanovskii} and \cite{ChristopherSubspace}. In \cite{zoladekCorrection} \zoladek gives a list of $52$ families of differential
forms known to have a center. 

Our main tool is a statistical method of Schreyer \cite{irred} to estimate the number
of components of the locus $Z_i$ where the first $i$ focal values vanish. 
The basic idea is to reduce the
equations $s_k$ modulo a prime number $p$ and count the number
of $\Fp$-rational points of $Z_i$ with a tangent spaces of fixed codimension.
By the Weil Conjectures \cite{weilconjectures}, which were proved by Delinge \cite{deligneproof}, we know that the fraction of points 
\[
	\gamma_p(Z_i^c) := \frac{\#\{\text{$\Fp$ rational points on $Z_i$  with 
	$\codim T_{Z,z} = c$}\}}
					{p^{14}}
\]
is equal to 
\[
	r\Bigl( \frac{1}{p} \Bigr)^c + \text{higher order terms}
\]
for a disjoint union of $r$ smooth codimension $c$ varieties. If the components of $Z_i$ are not smooth and disjoint, this number
is expected to be smaller. More precisely, if the set of singular points has $r_s$ components of
codimension $c_s$ in the codimension $c$ components of $Z$, we expect by the
same reasoning that the fraction of singular points
\[
	\gamma_p(\sing(Z_i^c)) := \frac{\#\{\text{$\Fp$ rational singular points on $\codim c$ components of $Z_i$}\}}
					{p^{14}}
\]
is equal to
\[
	r_s\Bigl( \frac{1}{p} \Bigr)^{c+c_s} + \text{higher order terms}.
\]
If $r_s$ is small with respect to $p^{c_s}$ this error does not change the
expected number $\gamma_p(Z_i^c)$ significantly. 

Instead of evaluating the $s_k$ at all possible points, we look at a large number of
random points and obtain an approximate value of $\gamma_p(Z_i^c)$ that can be used to 
estimate $r$ and therefore give an indication of the number of components in codimension
$c$. All this is reviewed in Section \ref{sHeuristic}.

In Section \ref{sExperiment} we apply the above 
method using our implementation of Frommers Algorithm. The resulting
estimates can be found in Figure \ref{fJakob}.

In Section \ref{sZoladek} we analyse \zoladeks families in detail. We choose random points
on each family and apply the same statistic as above. Here we find that 
most families are either parametrizing non-reduced components of $Z$ or 
subvarieties of true components. Only 22 families seem to parametrize reduced
components of $Z$. Those components can be found in Figure \ref{fZoladek}.

Comparing this to our estimate from Section \ref{sExperiment} we find that
up to codimension $7$ both counts agree. In codim $8$ we found heuristic evidence
for $4$ components in Section \ref{sExperiment} while in \zoladeks list we find
$5$ such components. This apparent contradiction is resolved by showing that
two of \zoladeks codimension $8$ families ($CR_4$ and $CR_6$) contain the same
differential forms. For codimension $9$, $10$ and $11$ the heuristic method
predicts many more reduced components than those that are contained in
\zoladeks lists. We therefore conjecture that there are many more components
to be discovered (see Conjecture \ref{cNumComponents}).

The computations for this article were done at the Gauss Laboratory at the University of G\"ottingen. 
The source code for the \verb#Macaulay2#  calculations of Section \ref{sZoladek} is contained in 
 \verb#survey2.m2# using the packages
\verb#CenterFocus# and \verb#Frommer#. 
These files and the source code for our C++ Implementation of Frommers Algorithm  can be found at \cite{centerfocusSourceforge}.
\verb#Macaulay2# is available at \cite{M2}.

%%%%%%%%%%%%%%%%%%%%%%%%%%%%%%%%%%%%%%%%%%%%
\section{Preliminaries} \xlabel{sPrelim}
%%%%%%%%%%%%%%%%%%%%%%%%%%%%%%%%%%%%%%%%%%%%
\nosubsections

If not stated otherwise we work over an algebraically closed field in this paper. 

We write the differential equation
\[
y' = - \frac{P(x,y)}{Q(x,y)}
\]
as $P(x,y)dx + Q(x,y)dy = 0$. 

\begin{notation}
Furthermore we denote by

\begin{tabular}{ll}
$V$ & the $20$ dimensional space of degree $3$ differential forms $Pdx+Qdy$. \\
$W$ & the $14$ dimensional subspace of Poincar\'e differential forms \\
	& $(x+P_2(x,y)+P_3(x,y))dx + (y + Q_2(x,y) + Q_3(x,y))dy$
%$K$ & a commutative ring (usually a field or $\ZZ$),\\
%$\PP^2:=\PP^2_K$ & the projective plane over this field, \\
%$K[x,y,z] $ & the coordinate ring of $\PP^2$,\\
%$\AZ^2 \subset \PP^2$ & the affine plane where $z\not=0$, \\
%$\lineinf \subset \PP^2$ & the line at infinity where $z=0$,\\
%$K[x,y,z]_d$& the vector space of homogeneous degree $d$ polynomials,\\
%$K[x,y,z]_d$& the vector space of homogeneous\\
%& $\quad$ degree $d$ polynomials,\\
%$P,Q \in K[x,y,z]_d$ & two such Polynomials,\\
%$Pdx + Qdy$ & the corresponding differential form on $\PP^2$,\\
%$\Vd \cong K[x,y,z]_d \oplus K[x,y,z]_d$ & the vector space of all such differential forms,\\
%$\Vd \cong K[x,y,z]_d $ & the vector space of all such differential forms,\\
%$\quad \quad \oplus K[x,y,z]_d$ \\
%$\{ x^iy^jz^{d-i-j}dx, x^iy^jz^{d-i-j}dy \}_{i+j\le d} $& the monomial basis of this vector space,\\
%$\{ x^iy^jz^{d-i-j}dx, $& the monomial basis of this vector space,\\
%$\quad x^iy^jz^{d-i-j}dy \}_{i+j\le d} $\\
%$p_{ij}, q_{ij} \in K$ & the coordinates of $\Vd$ with respect to this basis. \\
\end{tabular}
\end{notation}

\begin{defn}
The group
\[
	G := \Aff_2 := \left\{ \begin{pmatrix} M & v \\ 0 & 1\end{pmatrix} \, | \, \det M \not= 0 \right\} \subset GL(3)
\]	
with $M = \bigl( \begin{smallmatrix} m_{11} & m_{21} \\ m_{12} & m_{22} \end{smallmatrix} \bigr)$ and $v = \bigl( \begin{smallmatrix} v_1\\ v_2 \end{smallmatrix} \bigr)$
is called the {\sl affine linear group}. $G$ acts on the space of differential forms $V$ by
affine linear transfomations, i.e. for $g \in G$
\begin{align*}
	g\begin{pmatrix} x \\ y \end{pmatrix} &= M \cdot \begin{pmatrix} x \\ y \end{pmatrix} + v \\
	g\begin{pmatrix} dx \\ dy \end{pmatrix} &= M \cdot \begin{pmatrix} dx \\ dy \end{pmatrix} 
\end{align*}
The subgroup
\[
	O(2) := \left\{ \begin{pmatrix} M & 0 \\ 0 & 1\end{pmatrix} \, | \, M M^T =  \begin{pmatrix} 1 & 0 \\ 0 & 1\end{pmatrix}\right\} \subset G;
\]
is called the {\sl orthogonal group}. $O(2)$ acts on $W$ since it fixes the linear part $xdx + ydy = \frac{1}{2}D(x^2+y^2)$.
\end{defn}

\begin{defn}
A differential form $\omega = Pdx + Qdy \in V$ has a zero in $a$ if $P(a) = Q(a) = 0$. We say
that $Pdx + Qdy$ has a {\sl center} at $a$ if in addition there exist formal power series $\mu$ and $F$ centered at $a$ such that $\mu(a) \not=0$ and $dF = \mu \omega$. In this case $\mu$ is called an {\sl integrating factor} and $F$ a f{\sl irst integral}. 
%The subvariety $Z \subset V$ of all differential forms with a center is called the {\sl center variety}.
\end{defn}

\begin{lem}
If $\omega$ has a center at $a$ then $d\omega (a) = 0$.
\end{lem}

\begin{proof}
If $\omega$ has a center at $a$, there exist $\mu$ and $F$ with $dF = \mu \omega$
as above. Applying $d$ to this equation we obtain
\[
	0 = ddF = (d\mu)\omega + \mu (d\omega)
\]
Evaluating at $a$ yields
\[
	0 = (d\mu)(a)\omega(a) + \mu(a) (d\omega)(a) = \mu(a) (d\omega)(a)
\]
since $\omega(a)=0$. Now $\mu(a) \not=0$ by definition, so we obtain $(d\omega)(a) =0$.
\end{proof}

\begin{lem} \xlabel{lTaylor}
If a differential form $\omega$ has a center at $a$ then there exists a first integral $F$ at $a$
whose Taylor expansion at $a$ 
\[
	F = F_{a,0} + F_{a,1} + F_{a,2} + \dots
\]
satisfies $F_{a,0}=F_{a,1}=0$. 
\end{lem}

\begin{proof}
$\omega$ has a first integral since it has a center at $a$. Since in the definition
of first integral only $dF$ appears one can set $F_{a,0}=0$ without loss of generality. 
Now $F_{a,i}$ are homogeneous polynomials of degree $i$ in $(x-a_x)$ and $(y-a_y)$.
Therefore $dF_i(a) = 0$ for all $i \not= 1$. Now $F_1 = \alpha (x-a_x) + \beta (y-dy)$ for certain
$\alpha$ and $\beta$. We obtain
\[
	\alpha dx + \beta dy = dF_1(a) = dF (a) = (\mu\omega)(a) = 0
\]
and conclude $F_1 = 0$. 
\end{proof}

\begin{defn} 
In the situation of Lemma \ref{lTaylor}, $F_{2,a} =: F_2(\omega,a)$ is called the {\sl quadric associatied to $\omega$ in $v$}. The rank of $F_2(\omega,a)$ is invariant under affine coordinate tranformations.
\end{defn}

\begin{rem}
If $\omega$ has a center at $(0,0)$ and $\rank F_2(\omega,(0,0)) = 2$ we can assume that $F$ has no constant or linear terms as above. Over an algebraically closed field we can find a coordinate change such that 
\[
	F = \frac{1}{2} (x^2 + y^2) + \dots
\]
and $\omega = xdx + ydy + \dots$, i.e $\omega$ is a Poincare differential form. 

Over an arbitrary field this is only possible if additional conditions are satisfied. For example over $\RR$ one must assume that the quadratic form associated to $F_2$ is positive definite. 
\end{rem}

\begin{defn}
Let $Pdx + Qdy$ be a Poincar\'e differential form of degree $3$ over a field of characteristic $0$. One can then use Frommer's algorithm  to find a formal 
power series $F \in K[[x,y]]$ with
\[
	\det \begin{pmatrix} F_x(x,y) & F_y(x,y) \\ P(x,y,1) & Q(x,y,1) \end{pmatrix} = \sum_{j=1}^\infty s_j(P,Q)(x^{2j+2}+y^{2j+2}).
\]
In this situation $s_j(P,Q)$ is called the {\sl $j$th focal value} of $Pdx+Qdy$. Frommer's algorithm also
implies that $s_j$ is polynomial on $W$ and has rational coefficients. We call $s_j \in \QQ[p_{ij},q_{ij}]$ the {\sl $j$th focal
polynomial}.
\end{defn}

\begin{rem}
By analysing Frommer's Algorithm \cite{martin} one can show no prime factor of that the denominator of $s_j$
is bigger than $2j+2$. Therefore $s_j \mod p$ is well defined for $j \le (p-3)/2$.
\end{rem}

\begin{defn}
We define the ideals 
\[
	I_j = (s_1,\dots,s_j), \quad I_\infty = (s_1,s_2,\dots)
\]
and their vanishing sets $Z_j = V(I_j) \subset W$. $Z_\infty$ is a variety whose points
are exactly the Poincar\'e differential forms with a center at $(0,0)$. We therefore call
it the {\sl center variety}.
\end{defn}

\begin{rem}
In the case of degree 3 differential forms considered here, $\QQ[p_{ij},q_{ij}]$ has
14 variables. Hilbert's Nullstellensatz implies that $I_\infty$ can be generated by
finitely many elements, therefore there exist a number $m := m(3)$ such that
$Z_{\infty} = Z_m$ and $Z_\infty \not= Z_{m-1}$. The precise value of $m(3)$ is unknown. In \cite{focalValues} the inequality $m(3) \ge 13$ is proven for complex centers. Since one can not
study $Z_\infty$ explicitly we analyze $Z_{13}$ in this paper. If $m(3)=13$ this
is equivalent to analyzing $Z_\infty$. Otherwise we have $Z_\infty \subset Z_{13}$. 
\end{rem}

%%%%%%%%%%%%%%%%%%%%%%%%%%%%%%%%%%%%%%%%%%%%
\section{Finite Field Heuristics} \xlabel{sHeuristic}
%%%%%%%%%%%%%%%%%%%%%%%%%%%%%%%%%%%%%%%%%%%%
\nosubsections

In this section we explain how one can obtain heuristic information about a variety $X \subset \AZ^n$
by evaluating its defining equations at random points. For an extended discussion about this method see \cite{irred} or \cite{FiniteFieldExperiments}. An application of this method
to the Poincar\'e center problems in some solved and some unsolved
cases is described in \cite{zentrum}.

\begin{defn}
Let $X \subset \Anp$ be an algebraic variety. Denote the number of $\Fp$-rational
points of $X$ by $|X(\Fp)|$.  Then
\[
	\gamma_p(X) = \frac{|X(\Fp)|}{|\Anp|}
\]
is called the {\sl fraction of $\Fp$-rational points of $X$ in $\AZ^n$}.
\end{defn}

\begin{rem}
If $X$ has $r$ irreducible reduced smooth components of codimension $c$ and all other irreducible components
have larger codimension then the Weil-Conjectures imply that
\[
	\gamma_p(X) = r \Bigl( \frac{1}{p} \Bigr)^c + \text{higher order terms in $\frac{1}{p}$}
\]
\end{rem}

We will estimate $\gamma_p(X)$ statistically by evaluating the equations defining $X$ in a
number of randomly chosen points.

\begin{defn}
Let $X \subset \Anp$ be an algebraic variety. For a sequence 
$S= (x_1,\dots,x_N)$ of $\Fp$-rational points in $\Anp$ we call
\[
	\gammatilde_p(X,S) = \frac{| \{ i \suchthat x_i \in X\}|}{N}
\]
the {\sl empirical fraction of  $\Fp$-rational points}.
\end{defn}

\begin{rem}
The distribution of $\gammatilde_p(X,S)$ on the set of all sequences $S$ of length $N$
is binomial with mean $\mu(\gammatilde_p(X,S)) = \gamma_p(X)$ and standard deviation
\[
	\sigma(\gammatilde_p(X,S)) = \sqrt{\frac{\gamma_p(X)(1-\gamma_p(X))}{N}}
	\approx \sqrt{\frac{\gamma_p(X)}{N}}
\]
\end{rem}

This allows us to obtain an estimate of $\gamma_p(X)$ and then of $r$ and $c$ by evaluating
the equations of $X$ in many random points. More information is obtained, if we also calculate the
tangent space of $X$ in these random points: 

\begin{defn}
Let $X \subset \AZ^n$ be an algebraic variety defined by $f_1 =\dots = f_r = 0$. Then the tangent space of $X$ in a point $x \in X$ is defined as
\[
	T_{X,x} = \ker \Bigl(\frac{df_i}{dx_j}(x)\Bigr)_{i = 1 \dots r, j = 1 \dots n}.
\]
\end{defn}

\begin{rem} 
Let $X' \subset X \subset \AZ^n$ be an irreducible component, 
$x \in X'$ a point and $T_{X',x}$ the
tangent space of $X'$ in $x$. Then
\[
		\codim X' \ge \codim T_{X',x}
\]
with equality for general points if $X'$ is reduced. We therefore consider only points with $\codim T_{X',x} = c$ in estimating the number of components of codimension $c$. By the inequality
above we disregard all points on components of codimension greater then $c$.
\end{rem}

These arguments lead us to

\begin{heu} \xlabel{hTangent}
Evaluate the equations of $X$ in $N$ random points $x_i$ over $\FF_p$ and calculate the
tangent spaces $T_{X,x_i}$ in these points. Then estimate
\[
	\#\{ \text{reduced codim $c$ components} \} \approx \frac{\#\{i \suchthat \codim T_{X,x_i} = c\}}{N}p^c
\]
with an estimated error
\[
	\Phi \frac{\sqrt{\#\{i \suchthat \codim T_{X,x_i} = c\}}}{N}p^c.
\]
In this paper we have used $\Phi=2$ to obtain a confidence level of approximately $95\%$.
\end{heu}

\begin{cau}
Let $X^c$ be the subvariety of $X$ whose points have a tangent space of codimension $c$. Then above heuristic means that statistically the hypothesis $\gamma_p(X^c) = r (1/p)^c$ can not be rejected 
with confidence of more than $4.6 \%$. Algebraically this proves nothing, but gives a way to arrive at a reasonable conjecture about $X$. 
\end{cau}

\begin{cau}
It is possible that $X$ contains a component $Y$ that is irreducible over $\QQ$ but
decomposes into several irreducible components $Y_1, \dots, Y_k$ 
over the algebraic closure $\overline{\QQ}$, i.e.
the Galois group $\Gal(\overline{\QQ} / \QQ)$ acts transitively
on the $Y_i$. Over a finite field a $Y_i$ is rational if the Frobenius endomorphism
fixes $Y_i$. The expected number of such components is $1$. Therefore
our heurisic is an indication of the number of reduced irreducible components over
$\QQ$ and not over $\overline{\QQ}$ or $\CC$. 
\end{cau}

\begin{rem}
If the components of $X$ are not smooth and disjoint, then
\[
	\frac{\#\{i \suchthat \codim T_{X,x_i} = c\}}{N}p^c
\]
is expected to be smaller than the actual number of reduced 
codim $c$ component. More precisely, if the set of singular points has $r_s$ components of
codimension $c_s$ in the codimension $c$ components of $X$, we expect by the
same reasoning that the number of singular points on codim $c$ components
to be approximately
\[
	\frac{r_sN}{p^{c+c_s}}.
\]
If $r_s$ is small compared to $p^{c_s}$ our Heuristic \ref{hTangent} is therefore also useful in the presence of singularities.  If not, the number calculated can still be used as a heuristic lower bound on the number of 
reduced components.
\end{rem}

%%%%%%%%%%%%%%%%%%%%%%%%%%%%%%%%%%%%%%%%%%%%
\section{Experiments} \xlabel{sExperiment}
%%%%%%%%%%%%%%%%%%%%%%%%%%%%%%%%%%%%%%%%%%%%
\nosubsections

Using the heuristics described in Section \ref{sHeuristic} one can estimate the
number and codimension of reduced components of the center variety $Z_{\infty}$. For this we
study $Z_{13} \supset Z_{\infty}$ as an approximation. 
This is possible because Frommer's algorithm \cite{Frommer}, \cite{zentrum}, \cite{Moritzen} provides
a fast way to calculate the focal values of a given Poincar\'e differential
form even though the explicit polynomial expressions for the focal values
are not known. 

\begin{experiment}
We examined $402376372880300032 \cong 4.02 \times 10^{17}$  points over $\FF_{29}$ and
determined the rank of the Jacobi matrix if the first 13 focal values vanished using
our implementation of Frommers algorithm \cite{centerfocusSourceforge}.

This would take about 11 years of CPU time on a 2.3 GHz AMD Opteron Prozessor with 128 KB 
L1-Cache and 512 KB L2-Cache using our newes implementation of Frommers algorithm and
a parametrization for the solution set of the first three focal values to speed up the process.
We distributed the work to 56 processors. 

The heuristic estimate derived from this
experiment is shown in Figure \ref{fJakob}. Interesting differential forms
found in this and other computer experiments as well as statistics about
theses experiments are collected in our online
database \cite{centerfocusDatabase}. 
\end{experiment}

\begin{rem}
To test our implementation we have used it to recalculate the focal values of the examples in \cite{Hoehn}. Also the focal values of our example in \cite{focalValues}  were calculated independently by Colin Christopher using \verb#Reduce# and agree with ours modulo $29$. Furthermore the fact that for most \zoladek differential forms we indeed find points whose first 13 focal values vanish (see Section \ref{sZoladek}) can be interpreted as another test of our implementation. 

To test the parametrization of the first thee focal values we compare the results obtained with and without using parametrization.

To ensure that our experiments can be repeated we use a pseudo random number generator and store the svn revision number of the program version used to do the calculation in our database.
\end{rem}

\begin{rem}
By applying elements of the group $O(2)$ to a given differential form $\omega$ over $\FF_{29}$
we obtain further differential forms that have exactly the same properties as $\omega$. 
Now the group $O(2)$ has $2\cdot28^2$ elements over $\FF_{29}$ and therefore only approximately
$$\frac{29^{14}}{2 \cdot 28^2}  \approx 1.9 \times 10^{17}$$
fundamentally different differential forms exist over $\FF_{29}$. Since we
choose our points randomly it can happen, that some forms that are
equivalent with respect to $O(2)$ have been analysed 
several times. This makes no difference for our statistics, but prevents us
from looking at all points even though we have made more than $1.9 \times 10^{17}$
calculations. More precisely the propability of missing a general $O(2)$ orbit was
\[
	\left( 1- \frac{1}{1.9 \times 10^{17}}\right)^{4.0 \times 10^{17}} \approx
	\exp\left(-\frac{4.0 \times 10^{17}}{1.9 \times 10^{17}}\right) \approx 12\%
\]
for our experiment. Therefore one can expect that we have seen about $88\%$ 
of the fundamentally different differential forms. 
\end{rem}

\begin{figure}
\[
\begin{array}{|c|c|c|c|}\hline
\rank & \text{points found} & \text{estimated number} & \text{error} \\ 
 &  &   \text{of components} &  \\ \hline
1 & 208 & 0 & < 0.01\\
2 & 61435 & 0 & < 0.01\\
3 & 2506200 & 0 & < 0.01\\
4 & 27367779 & 0 & < 0.01\\
5 & 19681046795 & 1.00 & < 0.01\\
6 & 1328814108 & 1.96 & < 0.01\\
7 & 89629060 & 3.84 & < 0.01\\
8 & 3082816 & 3.83 & < 0.01\\
9 & 332067 & 11.97 & .04\\
10 & 31422 & 32.85 & .37\\
11 & 2556 & 77.50 & 3.07\\
12 & 1 & .88 & 1.76\\
\hline
\end{array}
\]
\caption{Number of Poincar\'e differential forms whose first 13 focal values vanish over $\FF_{29}$ found after shifting through 402376372880300032 random differential forms.}
\label{fJakob}
\end{figure}

%%%%%%%%%%%%%%%%%%%%%%%%%%%%%%%%%%%%%%%%%%%%
\section{\zoladeks Lists} \xlabel{sZoladek}
%%%%%%%%%%%%%%%%%%%%%%%%%%%%%%%%%%%%%%%%%%%%
\nosubsections

In \cite{zoladekRational} and \cite{zoladekCorrection} \zoladek has given a list of 52 families 
\[
	\phi \colon \AA^n \to V
\]
of degree  3 differential forms with a center. They are 
divided into 17 rational reversible systems  and 35 Darboux
integrable systems but this distinction is not needed for
our survey. No claim on the completeness of this list is made.

\begin{rem}
In \cite{zentrum} we proved that \zoladeks families $CR_5$ and $CR_7$ are
subfamilies of $\overline{CD_4}$ and similarily $CR_{12}$ and  $CR_{16}$ are subfamilies
of $\overline{CD_2}$.  We will therefore not consider them in this paper.
\end{rem}

\begin{rem}
Notice that the trivial Hamiltonian component of differential forms $\omega$ that satisfy
\[
	\omega = dF
\]
for a polynomial $F$ of degree $4$, is not on \zoladeks list.
\end{rem}

\begin{rem}
In printing long lists of polynomials it is impossible not to introduce missprints.
For this paper we have started from the implementation in Ulrich Rheins thesis \cite{Rhein}
and made further corrections. Together we have made the following changes
\begin{itemize}
\item  In  $CR_1$ we changed the
second occurrence of $q$ to a new variable, enlarging the family to all symmetric forms.
\item For $CR_3$ we changed $(2a^2-b)$ to $(2a-b^2)$ in the expression for $F$ and
$5a/2$ to $3a$ in the expression for $H$. The first change is in \cite{Rhein} the second isn't.
\item For $CR_5$, $\dot{x}$ and $\dot{y}$ have to be exchanged in \cite{zoladekCorrection}.
Also $l$ has to be changed to $ly$. This is already corrected in \cite{Rhein}. 
\item In $CR_{11}$ a sign mistake was introduced in \cite{Rhein}
\item For $CR_{17}$ the derivatives $\eta_x$ and $\eta_y$ were not calculated correctly in \cite{Rhein}.
\item For $CD_{17}$ the equation 
\[ 4\beta(\beta - 1)a^2 + 4\beta(3 - 2\beta)a + (3 - 2\beta)(1 - 2\beta) = 0 \]
must be satisfied. Fortunately the curve defined by this equation is rational and 
can be parametrized by
\begin{align*}
	 a &= \frac{-t_0^4+2t_0^3t_1-2t_0^2t_1^2-2t_0t_1^3+3t_1^4}{-2t_0^3t_1+8t_0^2t_1^2-14t_0t_1^3+12t_1^4}\\
	 \beta &= \frac{-6t_0t_1^3+12t_1^4}
{-2t_0^3t_1+8t_0^2t_1^2-14t_0t_1^3+12t_1^4}.
\end{align*}
We substituted this parameterization into the expression for $CD_{17}$ set $t_1=1$
and considered only the numerator of the resulting expression. The  parameterization
was kindly computed for us by Janko B\"ohm \cite{Parametrization}.
\item For $CD_{24}$ we did not find any centers over $\FF_{29}$
\item In $CD_{25}$ the coefficient of $x^3$ was changed from $a$ to $\alpha$. This misprint
was already corrected in \cite{Rhein}
\item In $CD_{26}$ the division $/2$ must be erased. This was also found by \cite{Rhein}.
\item For $CD_{32}$ we did not find any centers over $\FF_{29}$
\item From  $CD_{33}$ we obtain degree $4$ differentials for generic coefficients.
Only in the case $a=1$ we were able to factor out another factor $x$.
We therefore only use $CD_{33}$ with this additional restriction. 
\item Some families can be trivially enlarged by scaling with a nonzero scalar. We 
did this for all $CD$'s except $CD_5$ and $CD_8$ by multiplying the formula given
by \zoladek with the variable $aa_{16}$.
\end{itemize}
The families we used are contained in our \verb#Macaulay2# 
package \verb#CenterFocus# \cite{centerfocusSourceforge}, where
we have renamed the variables $a,\dots, t, \alpha, \beta, \gamma$ to $aa_{1},\dots,aa_{19}$.
\end{rem}

To estimate what part of our statistic in Figure \ref{fJakob} is explained
by \zoladeks examples we need to take into account, that
\zoladeks examples are general degree 3 differential forms in $V$ while 
we are interested in Poincar\'e differential forms in $W$. Over
an algebraically closed field every degree 3 differential form
$\omega$ with a non degenerate center can transformed into
a Poincar\'e differential form by an affine transformation. It is 
the purpose of this section to formalize this process and keep
track of the dimensions of the families involved. 
 
\begin{defn}
The affine linear group $G$ acts on the center variety. Therefore if
\[
	\phi \colon \AA^n \to V
\]
is a family of differential forms with a center, then
\begin{align*}
	\psi \colon G \times \AA^n &\to V\\
	(g,a) & \mapsto (g(\phi (a)))
\end{align*}
is a (possibly larger) family of differential forms with a center that is
invariant under action of $G$. Furthermore
\[
	\impsi \cap W
\]
is a variety of Poincar\'e differential forms. 
\end{defn}

\begin{rem}
$\impsi \cap W$
can have several components $W_i$ of which at least one contains differential
forms with a center at $(0,0)$. The subset of such differential forms is then 
dense inside this component.
\end{rem}

\begin{lem} \xlabel{lRank}
Let $\phi \colon \AA^n \to V$ be a morphism, and $D \phi$ its differential. 
If $a \in A$ is a integral point with $\rank (D\phi)(a))_{\FF_p} = n$ then
$\dim \phi$ = n.
\end{lem}

\begin{proof}
We have the following inequalities
\[
	n \ge \rank (D\phi)(a) \ge \rank ((D\phi) (a))_{\FF_p}  = n.
\]
Since the $D\phi$ drops rank only on Zariski closed subsets of $\AA^n$
we also know that $D\phi$ has generically rank $n$. If follows that $\phi$ is generically
locally injective and therefore $\dim \Img \phi = n$.
\end{proof}

\begin{calc} \xlabel{cDimFiberPhi}
We compared the number of variables $n$ involved in the definition of \zoladeks families 
with the rank of $D \phi$ in a random point $a$ 
using our script \verb#rankDifferential#. For all families both numbers agreed. Figure \ref{fDimImgPhi} contains \zoladeks families sorted by $n = \rank D\phi = \dim \Img \phi$.
\end{calc}

\begin{figure}
\[
\begin{array}{|c|c|c|}\hline
n & CR & CD \\ \hline
10 & 1 &   \\
8 & 6 , 11 &   \\
7 & 4 & 3 \\
6 & 2 , 14 & 1 , 2 , 4 \\
5 & 3 , 8 , 9 , 10 , 13 , 15 & 6 , 7 \\
4 &   & 8 , 13 , 14 , 18 , 19 , 20 , 21 , 28 , 34 , 35 \\
3 &   & 5 , 9 , 10 , 15 , 16 , 17 , 22 , 23 , 25 , 27 , 30 , 33 \\
2 & 17 & 11 , 12 , 24 , 26 , 29 , 31 , 32 \\
\hline
\end{array}
\]
\caption{Dimension of $\Img \phi$ for \zoladeks families} \label{fDimImgPhi}
\end{figure}

For the remaining calculations we need the following theorem on the
dimension of fibers of a morphim:%\semi:

\begin{thm} \xlabel{tSemi}
Let $\phi \colon X \to Y$ be a morphism of irreducible varieties over an algebraically closed
field. Then
\[
	\dim \phi^{-1}(y) \ge \dim X - \dim Y
\]
for all $y \in Y$
and there exist a Zariski open subset $U \subset Y$ such that
\[
	\dim \phi^{-1}(u) = \dim X - \dim Y
\]
for all $u \in U$. In this situation $\dim X - \dim Y$ is called the {\sl generic fiber dimension}.
\end{thm}

\begin{proof}
\cite[¤8, Theorems 2+3]{mumfordRed}
\end{proof}

\begin{defn}
For a family $\phi$ we denote by $d_1 = n - \dim \Img \phi$ the generic
fiber dimension of $\phi$. For all \zoladek families considered in this
paper we have seen $d_1=0$ in Calculation \ref{cDimFiberPhi}.
\end{defn} 

\begin{defn}
Let $\phi \colon \AA^n \to V$ be a family of differential forms and $a \in \AA^n$ a
point. Then 
\[
	G_a = \{ g \in G | g(\phi(a)) \in \imphi \}
\]
is called the set of {\sl irrelevant elements} of $G$ with respect to $a$. Consider
now the variety
\[
	X = \{(g,a) | g(\phi(a)) \in \imphi\} \subset G \times \AA^n
\]
and the projection
\[
	\pi \colon X \to \AA^n
\]
then $\pi^{-1}(a) = G_a$. From \semi\, we obtain that for almost all $a$ 
\[
	\dim G_a  = \dim X - n =: d_2.
\]
we call $d_2$ the generic dimension of $G_a$.	
\end{defn}

\begin{calc}
In Figure \ref{fDimGa} we list subsets $H_a \subset G_a$ for almost all $a$ for all
of \zoladek rationally reversible families. That these are indeed subsets is checked by
our script \verb#isIrrelevant#.
\end{calc}

\begin{figure}
\[
\begin{array}{|c|c|c|c|c|c|}\hline
\text{Familiy} 
	& H	& \dim H 	& \dim G_a 		& n 	& \codim \\ 
         	&	&         	&           			&  	& \Img \psi \cap W\\ 
	&  	&		& (\ref{cDimGa}) 	& (\ref{cDimFiberPhi})	& (\ref{calcDimWImgPsi})\\
	\hline
CR_{1} 
                    & \bgroup\begin{pmatrix}\text{$m_{11}$}&
0\\
0&
\text{$m_{22}$}\\
\end{pmatrix}\egroup\bgroup\begin{pmatrix}0\\
\text{$v_{2}$}\\
\end{pmatrix}\egroup & 3 & 3 & 10 & 6 \\ \hline CR_{2} 
                    & \bgroup\begin{pmatrix}\text{$m_{22}$}&
0\\
0&
\text{$m_{22}$}\\
\end{pmatrix}\egroup\bgroup\begin{pmatrix}0\\
0\\
\end{pmatrix}\egroup & 1 & 1 & 6 & 8 \\ \hline CR_{3 }
                    & \bgroup\begin{pmatrix}\text{$m_{11}$}&
0\\
0&
\text{i$m_{11}$}\\
\end{pmatrix}\egroup\bgroup\begin{pmatrix}0\\
0\\
\end{pmatrix}\egroup & 1 & 1 & 5 & 9 \\ \hline CR_{4} 
                    & \bgroup\begin{pmatrix}\text{$m_{11}$}&
\text{$v_{2}$} \text{$m_{11}$} {\text{$aa^{-1}_{3}$}}\\
0&
\text{$v_{2}$} \text{$m_{11}$} {\text{$aa^{-1}_{3}$}}+\text{$m_{11}$}\\
\end{pmatrix}\egroup\bgroup\begin{pmatrix}0\\
\text{$v_{2}$}\\
\end{pmatrix}\egroup & 2 & 2 & 7 & 8 \\ \hline CR_{5} & \multicolumn{5}{c|}{\text{subfamily of $CD_4$ \cite{zentrum}}} \\ \hline CR_{6} 
                    & \bgroup\begin{pmatrix}\text{$m_{11}$}&
\text{$v_{2}$} \text{$m_{11}$} {\text{$aa^{-1}_{3}$}}\\
0&
\text{$m_{22}$}\\
\end{pmatrix}\egroup\bgroup\begin{pmatrix}0\\
\text{$v_{2}$}\\
\end{pmatrix}\egroup & 3 & 3 & 8 & 8 \\ \hline CR_{7} & \multicolumn{5}{c|}{\text{subfamily of $CD_4$ \cite{zentrum}}} \\ \hline CR_{8} 
                    & \bgroup\begin{pmatrix}\text{$m_{22}$}&
0\\
0&
\text{$m_{22}$}\\
\end{pmatrix}\egroup\bgroup\begin{pmatrix}0\\
0\\
\end{pmatrix}\egroup & 1 & 1 & 5 & 9 \\ \hline CR_{9} 
                    & \bgroup\begin{pmatrix}\text{$m_{22}$}&
0\\
0&
\text{$m_{22}$}\\
\end{pmatrix}\egroup\bgroup\begin{pmatrix}0\\
0\\
\end{pmatrix}\egroup & 1 & 1 & 5 & 9 \\ \hline CR_{10} 
                    & \bgroup\begin{pmatrix}\text{$m_{22}$}&
0\\
0&
\text{$m_{22}$}\\
\end{pmatrix}\egroup\bgroup\begin{pmatrix}0\\
0\\
\end{pmatrix}\egroup & 1 & 1 & 5 & 9 \\ \hline CR_{11} 
                    & \bgroup\begin{pmatrix}\text{$m_{11}$}&
\text{$v_{2}$} \text{$m_{11}$} {\text{$aa^{-1}_{3}$}}\\
0&
\text{$v_{2}$} \text{$m_{11}$} {\text{$aa^{-1}_{3}$}}+\text{$m_{11}$}\\
\end{pmatrix}\egroup\bgroup\begin{pmatrix}0\\
\text{$v_{2}$}\\
\end{pmatrix}\egroup & 2 & 2 & 8 & 7 \\ \hline CR_{12} & \multicolumn{5}{c|}{\text{subfamily of $CD_2$ \cite{zentrum}}} \\ \hline CR_{13 }
                    & \bgroup\begin{pmatrix}\text{$m_{22}$}&
0\\
0&
\text{$m_{22}$}\\
\end{pmatrix}\egroup\bgroup\begin{pmatrix}0\\
0\\
\end{pmatrix}\egroup & 1 & 1 & 5 & 9 \\ \hline CR_{14} 
                    & \bgroup\begin{pmatrix}\text{$m_{22}$}&
0\\
0&
\text{$m_{22}$}\\
\end{pmatrix}\egroup\bgroup\begin{pmatrix}0\\
0\\
\end{pmatrix}\egroup & 1 & 1 & 6 & 8 \\ \hline CR_{15} 
                    & \bgroup\begin{pmatrix}\text{$m_{22}$}&
0\\
0&
\text{$m_{22}$}\\
\end{pmatrix}\egroup\bgroup\begin{pmatrix}0\\
0\\
\end{pmatrix}\egroup & 1 & 1 & 5 & 9 \\ \hline CR_{16} & \multicolumn{5}{c|}{\text{subfamily of $CD_2$ \cite{zentrum}}} \\ \hline CR_{17} 
                    & \bgroup\begin{pmatrix}1&
0\\
0&
1\\
\end{pmatrix}\egroup\bgroup\begin{pmatrix}0\\
0\\
\end{pmatrix}\egroup & 0 & 0 & 2 & 11 \\ \hline 
\end{array}
\]
\caption{Calculating the codimension of $\Img \psi \cap W$ for \zoladeks rationally
reverible systems} \label{fDimGa}
\end{figure} 

\begin{calc} \xlabel{cDimGa}
For every rational reversible \zoladek family we calculate $\dim G_{a_0}$ for a random element
$a_0$ using our script \verb#idealIrrelevantElementsRandom#. The results can also be found in Figure \ref{fDimGa}.
 Since $\dim G_{a_0}$ is always bigger then the generic
dimension of $G_a$ we obtain for almost all $a \in \AA^n$:
\[
	\dim G_{a_0} \ge d_2 \ge \dim H_a
\]
From the Figure \ref{fDimGa} we see that for every \zoladek family these inequalities have to be
equalities and we can calclate $d_2$.
\end{calc}

\begin{calc}
For every Darboux integrable \zoladek family we calculate $\dim G_{a_0}$ for a random element
using our script \verb#idealIrrelevantElementsRandom#. We obtain that this
dimension is zero for all $CD_i$. Since
\[
	\dim G_{a_0} \ge d_2 \ge 0
\]
We obtain $d_2=0$ for these cases.
\end{calc}

\begin{prop} \xlabel{pDimImgPsi}
Consider $\psi \colon G \times \AA^n \to V$ as above. Then 
\[
	\dim \Img \psi = n + 6 - d_1 - d_2.
\]
\end{prop}

\begin{proof}
Let $\psi(g,a) = g(\phi(a))$ be a generic element of $\Img \psi$. Then the
fibere over this element is
\begin{align*}
	F &= \psi^{-1}g(\phi(a)) \\
	&= \{ (h,b) | h(\phi(b)) = g(\phi(a) \} \\
	&= \{ (h,b) | (\phi(b) = h^{-1}g(\phi(a)) \} \\
	&= \{ (\tilde{h},b) | \phi(b) = \tilde{h}(\phi(a)) \} \subset G_a \times \AA^n.
\end{align*}
Now consider the projection
\[
	\pi \colon F \to G_a
\]
For $\tilde{h} \in G_a$ we obtain
\begin{align*}
	\pi^{-1}(\tilde{h}) \
	&= \{ b | \phi(b) = \tilde{h}(\phi(a)) \} \\
	&= \{ b | \phi(b) = \phi(a') \} \\
	&= \phi^{-1}(a')
\end{align*}	
For generic $a$ and $\tilde{h}$ we therefore have 
\[
	\dim F = \dim G_a + \dim \phi^{-1}(a) = d_1 + d_2
\]
by \semi. Using \semi\, again for $\psi$ and generic $F$ we get
\[
	\dim \Img \psi = \dim (G \times \AA^n) - \dim F = n+6-d_1-d_2
\]
\end{proof}

\begin{prop} \xlabel{pDimImgPsiW}
Consider the variety
\[
	X = \{ (g,\omega) | g(\omega) \in W \} \subset G \times \Img \psi
\]
and $X_0 \subset X$ an irreducible component. Let
\begin{align*}
%	\sigma \colon X_0 &\to W \\
%		(g,\omega) &\mapsto g(\omega)\\
	\pi \colon 	X_0 &\to G \\
		(g,\omega) &\mapsto g
\end{align*}
be the natural projection. In this situation all fibers of $\pi$ are isomorphic and
$\pi^{-1} (id)$  is an irreducible component of  $\Img \psi \cap W$. Furthermore
this component has the dimension $\dim X_0 - 6$.
\end{prop}

\begin{proof}
$G$ operates on $X$ via $h(g,\omega) = (gh,h^{-1}\omega)$. Since
$G$ is irreducible it also acts on every component $X_0 \subset X$. 
With this operation $\pi^{-1}(h) = h^{-1}(\pi^{-1}(\id))$. This proves the first claim.
Now
\[
	\pi^{-1}(\id) = \{ \omega | \omega \in W \} \cap X_0.
\]
This proves the second claim. The third claim follows from \semi.
\end{proof}
 
 \begin{lem} \xlabel{lSymIdeal}
 Let $X$ be the variety considerend in Proposition \ref{pDimImgPsiW} and the
 natural morphism
 \begin{align*}
 	\eta \colon X&\to \AA^2 \times \Img \psi \\
		((\begin{smallmatrix} M & v \\ 0 & 1 \end{smallmatrix}),\omega) 
		& \mapsto (-M^{-1}v,\omega).	
\end{align*} 
Then 
\[
	\Img \eta = \{ (\tilde{v},\omega) | \omega(\tilde{v}) = d\omega(\tilde{v}) = 0\,\text{and}\, \rank F_2(\tilde{v},\omega) = 2 \}.
\]
\end{lem}

\begin{proof}
If $\omega' := (\begin{smallmatrix} M & v \\ 0 & 1 \end{smallmatrix})(\omega)$ lies in $W$,
it satisfies $\omega'(0) = d\omega'(0) = 0$ and $\rank F_2(0,\omega') = 2$. But then $\omega(-M^{-1}v) = d \omega(-M^{-1}v) = 0$. 
With $\tilde{v} := -M^{-1}v$ this shows
\[
	\Img \eta \subset \{ (\tilde{v},\omega) | 
	\omega(\tilde{v}) = d\omega(\tilde{v}) = 0
	\,\text{and}\, 
	\rank F_2(\tilde{v},\omega) = 2 \}.
\]
Conversely consider $(\tilde{v},\omega)$ with $\omega(\tilde{v}) = d\omega(\tilde{v}) = 0$ and $\rank F_2(\tilde{v},\omega)=2$. Then with $g = (\begin{smallmatrix} 1 & -\tilde{v} \\ 0 & 1 \end{smallmatrix})$ we have $\omega' = g(\omega)$ satisfiying $\omega' (0) = 0$. This shows that $\omega'$ is of the form
\[
	\omega' =   \omega_1' + \omega_2' + \dots
\]
with $\omega_1' = dF_2'$. Since $\rank F_2 = 2$  there exists an element 
$h = (\begin{smallmatrix} M & 0 \\ 0 & 1 \end{smallmatrix})$
such that $h(F_2')$ is $\frac{1}{2}(x^2+y^2)$. It follows that
\[
	\omega'' := h(\omega') = xdx + ydy + \dots
\]
and $(h \circ g, \omega)$ is an element of $X$ with image $(\tilde{v},\omega)$.
 \end{proof}
 
 \begin{figure}
\[
\begin{array}{|c|c|c|c|c|c|c|c|c|c|c|c|c|}\hline
 & 0 & 1 & 2 & 3 & 4 & 5 & 6 & 7 & 8 & 9 & 10 & 11 \\ \hline
CR_{1} &   &   &   &   &   &   & .87 &   &   &   &   &   \\
CR_{2} &   &   &   &   &   &   &   & .01 & .79 &   &   &   \\
CR_{3} &   &   &   &   &   &   &   &   &   & .74 &   &   \\
CR_{4} &   &   &   &   &   &   &   &   & .81 &   &   &   \\
CR_{6} &   &   &   &   &   &   &   &   & .88 &   &   &   \\
CR_{8} &   &   &   &   &   &   &   &   & .03 &   &   &   \\
CR_{9} &   &   &   &   &   &   &   &   &   & .73 &   &   \\
CR_{10} &   &   &   &   &   &   &   &   &   & .77 &   &   \\
CR_{11} &   &   &   &   &   &   &   & .73 &   &   &   &   \\
CR_{13} &   &   &   &   &   &   &   &   & .03 &   &   &   \\
CR_{14} &   &   &   &   &   &   &   &   & .66 & .09 &   &   \\
CR_{15} &   &   &   &   &   &   &   &   &   & .76 &   &   \\
CR_{17} &   &   &   &   &   &   &   &   &   &   & .03 &   \\
\hline
\end{array}
\]
\caption{Estimated number of center variety components
parametrized by \zoladeks rational reversible families} \label{fCRheuristic} 
\end{figure}

\begin{figure}
\[
\begin{array}{|c|c|c|c|c|c|c|c|c|c|c|c|c|}\hline
 & 0 & 1 & 2 & 3 & 4 & 5 & 6 & 7 & 8 & 9 & 10 & 11 \\ \hline
CD_{1} &   &   &   &   &   &   & .01 & .87 &   &   &   &   \\
CD_{2} &   &   &   &   &   &   &   & .76 &   &   &   &   \\
CD_{3} &   &   &   &   &   &   & .84 &   &   &   &   &   \\
CD_{4} &   &   &   &   &   &   &   & .77 &   &   &   &   \\
CD_{5} &   &   &   &   &   &   &   &   &   &   &   &   \\
CD_{6} &   &   &   &   &   &   &   & .03 &   &   &   &   \\
CD_{7} &   &   &   &   &   &   &   &   & .87 &   &   &   \\
CD_{8} &   &   &   &   &   &   &   &   &   & .93 &   &   \\
CD_{9} &   &   &   &   &   &   &   &   &   &   &   &   \\
CD_{10} &   &   &   &   &   &   &   &   &   &   & .96 &   \\
CD_{11} &   &   &   &   &   &   &   &   &   &   &   &   \\
CD_{12} &   &   &   &   &   &   &   &   &   &   &   &   \\
CD_{13} &   &   &   &   &   &   &   &   &   &   &   &   \\
CD_{14} &   &   &   &   &   &   &   &   & .02 &   &   &   \\
CD_{15} &   &   &   &   &   &   &   &   &   &   &   &   \\
CD_{16} &   &   &   &   &   &   &   &   &   & .02 &   &   \\
CD_{17} &   &   &   &   &   &   &   &   &   &   & .59 &   \\
CD_{18} &   &   &   &   &   &   &   &   & .02 &   &   &   \\
CD_{19} &   &   &   &   &   &   &   &   & .03 &   &   &   \\
CD_{20} &   &   &   &   &   &   &   &   & .02 &   &   &   \\
CD_{21} &   &   &   &   &   &   &   &   &   & .78 &   &   \\
CD_{22} &   &   &   &   &   &   &   &   &   &   &   &   \\
CD_{23} &   &   &   &   &   &   &   &   &   & .02 &   &   \\
CD_{25} &   &   &   &   &   &   &   &   &   &   & .64 &   \\
CD_{26} &   &   &   &   &   &   &   &   &   &   &   &   \\
CD_{27} &   &   &   &   &   &   &   &   &   &   & .67 &   \\
CD_{28} &   &   &   &   &   &   &   &   &   & .12 &   &   \\
CD_{29} &   &   &   &   &   &   &   &   &   &   &   &   \\
CD_{30} &   &   &   &   &   &   &   &   &   & .03 &   &   \\
CD_{31} &   &   &   &   &   &   &   &   &   &   &   & .84 \\
CD_{33} &   &   &   &   &   &   &   &   &   &   &   &   \\
CD_{34} &   &   &   &   &   &   &   &   &   &   &   &   \\
CD_{35} &   &   &   &   &   &   &   &   & .02 &   &   &   \\
\hline
\end{array}
\]

\caption{Estimated number of center variety components
parametrized by \zoladeks Darboux integrable families} \label{fCDheuristic} 
\end{figure}

 \begin{lem} \xlabel{lOtwo}
In the situation of Lemma \ref{lSymIdeal} the fibers of  
\[
 	\eta \colon X \to \Img \eta \subset \AA^2 \times \Img \psi
\]
are isomorphic (as varieties) to $O(2) \subset G$. In particular
the components $X_i$ of $X$ are in $1:1$ correspondence with the
components $E_i$ of $\Img \eta$ and $\dim X_i = \dim E_i + 1$.
\end{lem}

\begin{proof}
The group $G$ acts on $\AA^2 \times \Img \psi$ via
\[
	h(\tilde{v},\omega) = (h(\tilde{v}),h^{-1}(\omega))
\]
where for 
$h = (\begin{smallmatrix} M' & v'  \\ 0 & 1 \end{smallmatrix})$
we set
\[
	h(\tilde{v}) := -(M')^{-1}(v'+\tilde{v}).
\]
With this action the morphism $\eta$ is $G$ covariant. It follows that
the fiber $\eta^{-1}(v,\omega)$ is isomorphic to a fiber $\eta^{-1}(0,\omega')$
with $\omega' = h^{-1}(\omega)$ for an $h$ with $h(v) =0$. We have
\begin{align*}
	\eta^{-1}(0,\omega') 
	&= \{ M | 
	(\begin{smallmatrix} M & v  \\ 0 & 1 \end{smallmatrix})(\omega') \in W 
	\wedge -M^{-1} v = 0 \} \\
	&= \{ M | 
	(\begin{smallmatrix} M & 0  \\ 0 & 1 \end{smallmatrix})(\omega') \in W \} .
\end{align*}
This set is non empty, since $(0,\omega')$ is in the image of $\eta$. Therefore there exists an
$h'$ such that $h'(0,\omega') = (0,\omega'')$ with $\omega'' \in W$. Now
\begin{align*}
	\eta^{-1}(0,\omega'') 
	&= \{ M | 
	(\begin{smallmatrix} M & 0  \\ 0 & 1 \end{smallmatrix})(\omega'') \in W \} \\
	&= O(2) \subset G
\end{align*}
since only elements with $MM^T = (\begin{smallmatrix} 1 & 0  \\ 0 & 1 \end{smallmatrix})$
fix the linear part $\omega_1'' = xdx + ydy$. 
\end{proof}

\begin{cor} \xlabel{cDimWImgPsi}
If $\phi$ is a family of differential forms with a center whose generic element has
only finitely many zeros and a center of rank $2$, then 
\[
	\dim W \cap \Img \psi = n+1-d_1-d_2.
\]
\end{cor}

\begin{proof}
By assumption a Zariski open subset of $\Img \phi$ contains differential forms 
with a rank $2$ center. Since this fact is invariant under action of $G$ the same is true for
$\Img \psi$. For every such element  $\omega \in \Img \psi$ one can find an element $g \in G$
such that $g(\omega)$ is a Poincar\'e differential form in $W$. Therefore $\xi \circ \eta$ is
dominant. If $\omega$ has only finitely many zeros then $\xi^{-1}(\omega)$ is finite, so $\xi$
is generically finite by our assumptions. We obtain
\[
	\dim X = \dim \Img \psi +1.
\]
Using Proposition \ref{pDimImgPsi} and Proposition \ref{pDimImgPsiW} we obtain
\[
	\dim \Img \psi \cap W = \dim X - 6 = \dim \Img \psi -5 = n+1 - d_1 -d_2.
\]
\end{proof} 

\begin{calc} \xlabel{calcDimWImgPsi}
Using Corollary \ref{cDimWImgPsi} we calculate $\dim W \cap \Img \psi$ for \zoladeks families of rationally
revesible centers. The results
can also be found in Figure \ref{fDimGa}. For \zoladeks families of Darboux centers we have
$d_1=d_2=0$ and therefore the dimensions are equal to $n+1$ and can be read from Figure \ref{fDimImgPhi}.
\end{calc}

\begin{rem}
We collect the previous definitions, lemmata and propsitions in the following
diagram:
\xycenter{
	G 
	&
	& X_i \ar[ll]_{\pi}^{\txt{\tiny $\codim 6$ Fibers}} \ar@{{}{}{}}[r]|\subset \ar[d]
	& X \ar@{{}{}{}}[r]|\subset \ar[d]_{\eta}^{\txt{\tiny O(2)-Fibers}}
	& G \times \Img \psi \ar[d]
	\\
	&
	& E_i \ar@{{}{}{}}[r]|\subset \ar[d]
	& \Img \eta \ar@{{}{}{}}[r]|\subset \ar[d]_{\xi}^{\txt{{\tiny finite}}}
	& \AZ^2 \times \Img \psi \ar[d]
	\\
	&
	& \Img \psi \ar@{{}{}{}}[r]|=       
	& \Img \psi \ar@{{}{}{}}[r]|=
	& \Img \psi
	}
where labels of the $\xi$- and $\eta$ arrows denote the expected fibers. Special
fibers could have a different structure.
\end{rem}

To estimate the component structure of $\pi^{-1}(\id) = \Img \psi \cap W$ we use
again our heuristic approach. 

\begin{calc}
Starting from a family $\phi \colon \AA^n \to V$, we find rational points on $\Img \psi \cap W$ as follows. First choose a rational point $a \in \AA^n$ and consider the differential form
$\omega = \psi(\id,a) = (\phi(a)) \in \Img \psi$. If $v_1,\dots,v_k$ are the rational symmetric
zeros of $\omega$,
then the rational points in the preimage of $\xi$ are
\[
	\xi^{-1} (\omega) = \{ (v_i,\omega) \, | \, i=1,\dots,k \}.
\]
For each pair $(v_i,\omega) \in \AA^2 \times \Img \psi$ the preimage of $\eta$ is
\[
	\eta^{-1} (v_i,\omega) = \{(M,-Mv_i,\omega_g \, | \, (M,-M^{-1}v_i)(\omega) \in W\}
\]
Now 
\[
	(\begin{smallmatrix} M & -Mv_i  \\ 0 & 1 \end{smallmatrix}) = 
	(\begin{smallmatrix} M & 0  \\ 0 & 1 \end{smallmatrix})
	(\begin{smallmatrix} 1 & -v_i  \\ 0 & 1 \end{smallmatrix})
\]
and $\omega_i := (\begin{smallmatrix} 1 & -v_i  \\ 0 & 1 \end{smallmatrix})(\omega)$ has
zeros at $(v_j-v_i)$ in particular one at zero. Therefore $\omega_i$ is of the form
\[
	\omega_i = l_{11} xdx + l_{12}(xdy+ydx) + l_{22}ydy + \text{higher order terms}
\]
We have $(\begin{smallmatrix} M & 0  \\ 0 & 1 \end{smallmatrix})(\omega_i) \in W$
if and only if
\[
	M^t L_i M = (\begin{smallmatrix} 1 & 0  \\ 0 & 1 \end{smallmatrix})
\]
with
\[
	L_i = (\begin{smallmatrix} l_{11} & l_{12}  \\ l_{21} & l_{22} \end{smallmatrix}).
\]
Since by Lemma \ref{lOtwo} the solution is a one dimensional space,
we can fix one entry of $M$ and and generically obtain finitely many rational solutions
$M_{i1},\dots, M_{il}$. In this manner we have found finitely many rational points
\[
	(g,\omega) \in (\eta \circ \xi)^{-1}(\omega) \subset X \subset G \times \Img \phi
\]
with 
\[
	g \in \{ (M_{ij},-M_{ij}v_i) \}.
\]
To obtain points in $W \cap \Img \psi = \pi^{-1}(\id)$ we operate with $g^{-1}$ on the
whole situation, and obtain points
\[
	(\id, g(\omega)) \in X \subset G \times \Img \psi
\]
The points $g(\omega)$ can then be analysed with our implementation of Frommer's
algorithm. If the first $13$ focal values of $g(\omega)$ vanish we calculate the
codimension of the tangent space to $Z_{13}\supset Z_{\infty}$ in these points. 
The results of doing this
for $2000$ random choices of $a$ in each of \zoladeks \, families are available as hash tables  
\verb#experimentsCR# and \verb#experimentsCD# in \verb#survey2.m2#.
For the families $CD_{24}$ and $CD_{32}$ we did not find any differential forms
this way. 
\end{calc}

We now want to identify those \zoladek-families that define reduced components
of the center variety $Z_{13} \supset Z_{\infty}$. For this we use again our finite field heuristic.

\begin{calc}
Consider a family $\phi \colon \AA^n \to V$ and set $d=W \cap \Img \psi$.
$W$ might have several components $W_i$ of which at least one has dimension $d$.
By the procedure above we expect
to find approxemately $2000/p^{d-\dim W_i}$ points on a component 
$W_i$ of $W \cap \Img \psi$. The generic codimension of a tangent
space to the center variety in points of $W_i$ is $14-\dim W_i$ if and 
only if $W_i$ is also a reduced component of the center variety. We can 
therefore heuristically identify reduced components of the center variety $Z_{13}$
by scaling our point counts by $p^{d-14+c}/2000$ where $c$ is the codimension
of the tangent space a each point. The result is contained in Figures \ref{fCRheuristic} and \ref{fCDheuristic}.
\end{calc}

\begin{figure}
\[
\begin{array}{|c|c|c|}\hline
\text{codim} & CR & CD \\ \hline
6 & 1 & 3 \\
7 & 11 & 1 , 2 , 4 \\
8 & 2 , 4 , 6 , 14 & 7 \\
9 & 3 , 9 , 10 , 15 & 8 , 21 \\
10 &   & 10 , 17 , 25 , 27 \\
11 &   & 31 \\
\hline
\end{array}
\]
\caption{\zoladek families that heuristically
parametrize reduced components of the center variety} \label{fZoladek}
\end{figure}

\begin{rem}
A family $\phi$ will have all numbers calculated above close to zero if
one of the following holds
\begin{enumerate}
\item $\phi$ defines only a subfamily of a true component of the center variety and the codimension of
the family inside the component is at least one.
\item $\phi$ defines a non reduced component of the center variety 
\item The generic point of $\phi$ does not have a symmetric center. 
\end{enumerate}
We suspect that all three possibilities actually occur. The third case
can be easily detected by analysing a generic point. This shows that
families $CD_{33}$ and $CD_{34}$ are of this kind. Probably this is either
due to misprints introduced by us or by misprints
in \cite{zoladekRational} or \cite{zoladekCorrection}
that we were not able to find and correct.

To distinguish between the cases (1) and (2) is much more difficult.
\end{rem}

\begin{figure}
%{\bf \Large Reduced Components of the Center Varitey}
\includegraphics[width=14cm,trim=0cm 1cm 0mm 20mm, clip=true]{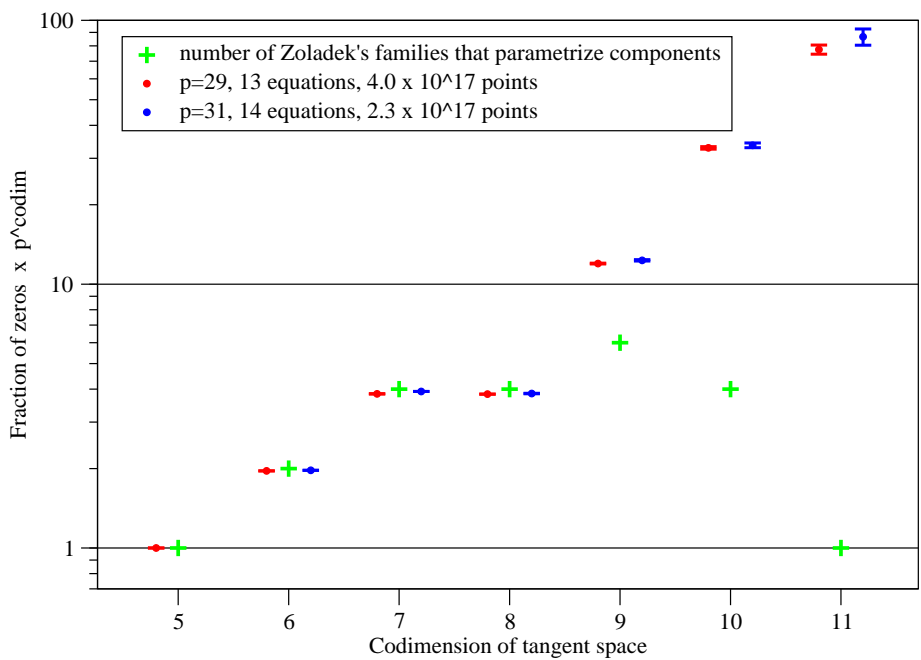}
\caption{Our heuristic predicts that up to codimension 8 all reduced 
components of the center variety are known. For higher codimension
many components are waiting to be discovered. \label{fAll}}
\end{figure}

\begin{rem}
Notice that only smooth points on each components have the
correct tangent dimenesion. Therefore we expect the results of the above
scaling to be less than $1$ for each component of the center variety. We have
collected those families that do parametrize a reduced component
of the center variety by this heuristic in Figure \ref{fZoladek}.
\end{rem}

Comparing the number of components contained in Figure \ref{fZoladek} with 
those of Figure \ref{fJakob} we find that up to codim 7
both counts agree in codim 8 there are 5 components given by \zoladek, while
we see only 4 in our heuristic. Fortunately Ulrich Rhein has found numerical
evidence for $CR_4 \subset CR_6$ in his Diploma Thesis \cite{Rhein}. It
is not difficult to prove that this is indeed the case:

\begin{prop}
All differentials parametrized by \zoladeks family $CR_4$ are
also contained in \zoladeks family $CR_6$.
\end{prop}

\begin{proof}
One can obtain $CR_4$ from $CR_6$ by
setting $k=0$ and renaming the variables as follows
$r \to q \to p \to n \to l \to k$ in \zoladeks notation.
\end{proof}
 
With this correction we have compared our heuristic component count
with the components detected among \zoladeks list in Figure \ref{fAll}.
We observe, that up to codimension $8$ both counts agree. Starting from
codimension $9$ there seem to exist many more reduced components
than previously known. We therefore

\begin{conj} \xlabel{cNumComponents}
The number of reduced components of the center variety in degree $3$ is
\begin{itemize}
\item 1 in codimension 5
\item 2 in codimension 6
\item 4 in codimension 7
\item 4 in codimension 8
\item at least 12  in codimension 9
\item at least 33  in codimension 10
\item at least 74 in codimension 11
\item possibly further components in codimension 12
\end{itemize}
\end{conj}

\begin{rem}
The \verb#Macaulay2#  calculations made in this section are contained in 
the file \verb#survey2.m2# using the packages
\verb#CenterFocus# and \verb#Frommer#. All three are available at \cite{centerfocusSourceforge}. \verb#Macaulay2# is available at \cite{M2}.
\end{rem}

\def\cprime{$'$} \def\cprime{$'$}

\end{document}